\documentclass[a4paper, 11pt]{article}
\usepackage[backend=biber,style=numeric,doi=false,isbn=false,url=false,eprint=true, sorting=nyt
]{biblatex}
\usepackage{setspace}
\usepackage{fullpage}
\usepackage{mathtools}
\usepackage[utf8]{inputenc}
\usepackage{amsthm,amsmath,amstext, amssymb, xcolor, scrextend, tikz, multirow, 
enumerate, wasysym, makecell, tabularx, hyperref,  pbox, lipsum,csquotes,  tikz-cd, 
multicol, calc, accents, stackengine, kpfonts}
\usepackage[all]{xy}
\usepackage{xcolor}
\usepackage{palatino}
\usepackage{cancel}
\usepackage{authblk}

\usepackage{mathtools}

\DeclareSymbolFont{AMSb}{U}{msb}{m}{n}
\DeclareMathSymbol{\N}{\mathbin}{AMSb}{"4E}
\DeclareMathSymbol{\Z}{\mathbin}{AMSb}{"5A}
\DeclareMathSymbol{\R}{\mathbin}{AMSb}{"52}
\DeclareMathSymbol{\Q}{\mathbin}{AMSb}{"51}
\DeclareMathSymbol{\I}{\mathbin}{AMSb}{"49}
\DeclareMathSymbol{\C}{\mathbin}{AMSb}{"43}

\DeclareFontFamily{U}{mathx}{\hyphenchar\font45}
\DeclareFontShape{U}{mathx}{m}{n}{<-> mathx10}{}
\DeclareSymbolFont{mathx}{U}{mathx}{m}{n}
\DeclareMathAccent{\widebar}{0}{mathx}{"73}

\def\lim{\mathop{\rm lim}\nolimits}

\def\Ext{\mathrm{Ext}}
\def\Hom{\mathrm{Hom}}

\def\ker{\mathop{\rm Ker}\nolimits}

\newcommand{\res}[1]{\hspace{-0.6mm}\downarrow_{\hspace{-0.25mm}{#1}}}
\newcommand{\dire}[1]{\displaystyle{\bigoplus_{{#1}}}}
\newcommand{\dbar}[2]{\displaystyle{\bigoplus_{{#1}}}\overline{#2}}
\newcommand{\dbn}[2]{\displaystyle{\bigoplus_{{#1}}}\overline{F({#2})_N}}

\newcommand{\F}{\mathbb{F}}

\DeclareMathOperator{\Ima}{Im}

\numberwithin{equation}{section}
\newtheorem{theorem}{Theorem}
\newtheorem{prop}[theorem]{Proposition}
\newtheorem{lemma}[theorem]{Lemma}

\newtheorem{remark}{Remark}
\theoremstyle{definition}
\newtheorem{defn}{Definition}
\newtheorem{example}{Example}

\usetikzlibrary{decorations.shapes}
\usetikzlibrary{decorations.text}
\usetikzlibrary{calc}

\hypersetup{colorlinks=true,linkcolor=blue,citecolor=blue, filecolor=blue,urlcolor=blue}
\title{A recognition theorem for permutation modules over  $p$-groups extending Weiss' Theorem}

\author[1]{Marlon Estanislau}
\affil{Universidade Federal de Minas Gerais, Belo Horizonte, MG, Brazil}
\begin{document}
\maketitle
\onehalfspacing
\footnotetext[1]{\textit{Email addresses:} mestanislau@ufmg.br}
\footnotetext[2]{Research is part of the author's PhD thesis at the Federal University of Minas Gerais, supervised by 
John William MacQuarrie.}
\begin{abstract}
    Let $G$ be a finite $p$-group with normal subgroup $N$, and $R$ a complete discrete valuation ring in mixed characteristic.  We characterize permutation $RG$-modules in terms of modules for $RN$ and $R[G/N]$.  The result generalizes both the seminal detection theorem for permutation modules due to Weiss, who characterizes those permutation $RG$-modules that are $RN$-free when $R$ is a finite extension of $\Z_p$, and a more recent result of MacQuarrie and Zalesskii, who prove a characterization of permutation modules when $N$ has order $p$ and $R = \Z_p$.
    
    \vspace{1.5mm}
    
    \noindent \textbf{keywords}: finite $p$-groups, permutation modules, lattices, $p$-adic representations

   \end{abstract}
\section{Introduction}
    Let $p>0$ be a prime number. Throughout this article, $G$ is  a finite $p$-group and $N$ is a normal subgroup of $G$ (denoted $N\lhd G$). By $R$ we denote a complete discrete valuation ring in mixed characteristic, that is, with residue field $\F$ of characteristic $p$ and fraction field of characteristic $0$.
    
    Let $A$ denote a quotient ring of $R$.
    An $AG$-\emph{lattice} is an $AG$-module that is $A$-free and finitely  generated. We are particularly interested in $AG$-lattices that are \emph{$AG$-permutation modules}: an $AG$-lattice having an $A$-basis that is preserved set-wise by the multiplication
    of $G$.  Such lattices are extremely well-behaved and serve as a natural generalization of free modules.  However, given
    a lattice, it may not be clear whether it is a permutation module. In 1988, A.\ Weiss provided a detection theorem for permutation $RG$-modules when $R = \Z_p$ \cite[Theorem 2]{WeissAnnals}.  Weiss later generalized his own result, allowing $R$ to be a finite extension of $\Z_p$ (cf.\cite{Wei}), and more recently J. MacQuarrie, P. Symonds and P. Zalesskii generalized it further, allowing $R$ to be a complete discrete valuation ring in mixed characteristic:
\begin{theorem}(Special case of \cite[Theorem 1.2]{MACQUARRIE2020106925})\label{Theo.MSZ}
     Let $R$ be a complete discrete valuation ring in mixed characteristic with 
residue field of characteristic $p$, let $G$ be a finite $p$-group and let $U$ be an $RG$-lattice. Suppose there is a normal subgroup $N$ of $G$ such that:
\begin{enumerate}
    \item the module $U$ restricted to $N$ is  $RN$-free,
    \item the submodule of $N$-fixed points $U^N$ is an $R[G/N]$-permutation module.
   
\end{enumerate}
 Then $U$ itself is an $RG$-permutation module.
\end{theorem}

Both Weiss' original result and its generalizations have diverse applications, being applied in number theory, profinite group theory and block theory, among others (cf.\cite{BOLTJE202070}, \cite{Florian}, \cite{LIVESEY202194}, \cite{Puig},  \cite{WeissAnnals}, \cite{P}).

  Whenever $U$ is an $AG$-module,
    the \emph{restriction} of $U$ to $N$ is denoted by $U\downarrow_N$. Weiss' Theorem cannot be applied to lattices for which $U\res{N}$ is not free.  If $U$ is a permutation module, then of
    course $U\res{N}$ is necessarily a permutation module, but, in general,  a sensible analogous statement to  Weiss' Theorem demanding only that $U\res{N}$
    be a permutation module seems far from obvious (the naive generalization holds when $G$ is cyclic
     \cite[Theorem 1.4]{M.cyclic} and \cite[Proposition 6.12]{TORRECILLAS2013533}, but not otherwise). To state  our result we require some definitions. 
    Whenever $U$ is $AG$-module, by $U^N$ (resp.\ $U_N$) we denote the 
    $N$-\emph{invariants} (resp.\ $N$-\emph{coinvari\-ants}) of $U$ -- that is, the largest $AG$-submodule (resp.
    $AG$-quotient module) of $U$ on which $N$ acts trivially.
    As usual, we denote by $\text{H}^1(G,-)$ the first right derived functor of the fixed point functor $(-)^G$.      
   An $RG$-module $U$ is called $G$-\emph{coflasque} if $\text{H}^1(L,U)=0$ for every subgroup $L$ of $G$.

    \begin{defn} We call an $RG$-module $U$ an $RG$-\emph{permutation block module} if 
    $$U=\dire{i\in I}B_i,$$
    where $B_i$ is an $A_iG$-permutation module with 
    $A_i=R/p^iR$ and $I$  a finite subset of $\Z_{\geq0}$.
    \end{defn}
    With this notation and definition, our main result is as follows: 
   
   \begin{theorem}\label{main}
    Let $U$ be an $RG$-lattice and $N\lhd G$. Then $U$ is an $RG$-permutation module if, and only if, the following conditions are satisfied:
    \begin{enumerate}
        \item $U\res{N}$ is an $RN$-permutation module,
        \item\label{it2} $U^N$ is $G/N$-coflasque, and $U_N$ is an $R[G/N]$-permutation module,
        \item\label{it3} $(U/U^N)_N$ is an $R[G/N]$-permutation block module.
    \end{enumerate}
    \end{theorem}

This result generalizes \cite[Theorem 2]{John}, who give an equivalent characterization in the special case when $R = \Z_p$ and $|N| = p$.
  The proof of Theorem \ref{main} uses a central idea from the proof of \cite[Theorem 2]{John}, but in general it is more complicated, requiring a more delicate analysis of the structure of $U\res{N}$ and techniques from  homological algebra. 
   The conditions in  Theorem \ref{main} are minimal, in the sense that if we do not suppose that $U\res{N}$ is a permutation module (Example \ref{exe}), or  that $U^N$ is $G/N$-coflasque \cite[Example 2]{John}, or  that $U_N$ is a permutation module \cite[Example 1]{John}, or  that $(U/U^N)_N$ is a permutation block module \cite[Examples 1 and 2]{J.M}, then $U$ may not be a permutation module.

\subsubsection*{Acknowledgements}
    The author thanks John William MacQuarrie for helpful conversations and for his attention and assistance in preparing this manuscript. The author is also grateful to Csaba Schneider for his help with computational methods used to build Example \ref{exe}. 
    The author was supported by  CAPES Doctoral Grant 88887.688170/2022-00 and by FAPEMIG Doctoral Grant 13632/2025-00.

 \section{Preliminaries}\label{s2} 
      Let $\pi$ be a prime element of $R$ and  denote by  $\F$  the residue field $R/\pi R$. Whenever  $U$ is an $RG$-module, we denote by $\overline{U}$ the $\F G$-module  $U/\pi U$. We will write at times 
   $\overline{X}$ when $X$ is a submodule of $U$.  Since $X$ will always  be an $R$-direct summand of  $U$, the two possible interpretations coincide: $(X+\pi U)/\pi U\simeq X/\pi X$. 
      \begin{lemma}\label{Ext=0}
    Let $U$ be a finitely generated $RG$-lattice and suppose that $\Ext_{RG}^1(U,U)=0$. If there is a decomposition of $\overline{U}=X'\oplus Y'$ as an $\F G$-module, then there exists a decomposition $U=X\oplus Y$ as an $RG$-module with  $\overline{X}=X'$ and $\overline{Y}=Y'$.
\end{lemma}
\begin{proof}
    The exact sequence
    $$\xymatrix{0\ar[r]&U\ar[r]^{\pi}&U\ar[r]&\overline{U}\ar[r]&0},$$
    induces the exact sequence 
    $$\xymatrix{0\ar[r]& \Hom(U,U)\ar[r]^{\pi}&\Hom(U,U)\ar[r]&\Hom(U,\overline{U})\ar[r]&\Ext_{RG}^1(U,U)}=0,$$
    so $\overline{\text{End(U)}}\simeq \text{End}(\overline{U})$. Let $\alpha'$ and $\beta'$ be the projections on $X'$ and $Y'$ respectively with $\alpha'+ \beta'=id_{\overline{U}}$.  Now, $\text{End}(U)$ is a finitely generated algebra over a complete discrete valuation ring, so we can find idempotents $\alpha$ and $\beta$ in $\text{End}(U)$ with $\alpha + \beta=id_U$  satisfying  $\overline{\alpha}=\alpha'$ and $\overline{\beta}=\beta'$ \cite[Theorem 1.9.4]{benson}.  Accordingly, $U$ is a direct sum of
 $RG$-modules $X \oplus Y$ with $\overline{X} = X',\overline{Y} = Y'$. 
\end{proof}

\begin{lemma}(\cite[Corollary 6.8]{MACQUARRIE2020106925}\label{permuExt=0}
    Let $G$ be a finite $p$-group. If $A$ and $B$ are $RG$-permutation modules then $\Ext_{RG}^1(A,B)=0$. 
\end{lemma}

   Let $N$ be a normal subgroup of $G$. Remember that the $N$-invariants of the $RG$-module $U$, the module $U^N$, is the largest submodule of $U$ such that $N$ acts trivially and  the $N$-inveriants of $U$, the module $U_N$, is the largest quotient module of $U$ such that $N$ acts trivially. Denote by $I_N$ the kernel of the augmentation map $RN\longrightarrow R$.  Explicitly, we have:
   $$U^N:=\{u \in U: nu=u \;\forall\; n\in N\},$$
   $$U_N:=U/I_NU.$$
 
      Let $S$ be the set of all subgroups of $N$ and let $\Gamma$ denote a set of representatives of the distinct  orbits of the action by conjugation of $N$ on $S$. From now on, we suppose that   $U$ is an $RG$-lattice such that $U\res{N}=\bigoplus_{H\in\Gamma}F(H)$ where $F(H)$ is an $RN$-permutation module with all its indecomposable summands isomorphic to $R[N/H]$. In what follows, we always suppose that $H$ is a subgroup of $N$ and $N$ is a normal subgroup of $G$. Observe that in this decomposition, the $RN$-direct summand $F(N)$ is an $RN$-lattice on which $N$ acts trivially. We say that $F(N)$ is a maximal trivial summand of $U\res{N}$. Furthermore, we will also  write $F(N)$ for the module $(F(N)+I_NU)/I_NU$, because, in this case $$U_N=\dire{H\in\Gamma\setminus \{N\}}F(H)_N\oplus F(N)_N\simeq\dire{H\in \Gamma\setminus \{N\}}F(H)_N\oplus F(N).$$
      
        Throughout the text, by $\varphi$ we will denote the natural projection $U\longrightarrow U_N$ and by  $\rho: U_N\longrightarrow (U/U^N)_N$ the natural projection induced by the natural projection $U\longrightarrow U/U^N$.
        
     \begin{prop}\label{tor}
        Let $V$ be the $RN$-lattice  $ R[N/H]$, and let $M$ be a normal subgroup of $N$. Then we have  the following isomorphisms of $RN$-modules:
        \begin{enumerate}
        \item $(V/V^N)_N\simeq R/p^kR$ where $p^k=[N:H]$.
        \item  $V_M\simeq R[N/MH]\simeq V^M$.
        \item $(V/V^M)_M\simeq (R/p^kR)[N/MH]$  where $p^k=[M:M\cap H ]$.
        \end{enumerate}
    \end{prop}
    \begin{proof}
        \begin{enumerate}
         \item   Since there exist $u\in V$ such that $$V^N=\left\langle\sum_{g\in N/H}gu\right\rangle_{R} \;\text{and}\; V_N=\langle u +I_NV\rangle_{R},$$ we have $$\varphi(V^N)=\left\langle p^k(u+I_NV)\right\rangle_R=p^kV_N,$$ where $p^k=[N:H]$. Thus, we obtain   $$R/p^kR=V_N/\varphi(V^N)=V/(I_NV+V^N)=(V/V^N)_N.$$
         \item Consider the surjective map $\beta: R[N/H]\to R[N/MH]$ given by $\beta(nH)=nMH$ for $n\in N$. We have $\ker \beta= I_MR[N/H]$, so $V_M\simeq R[N/MH]$. Now, by the Mackey decomposition formula \cite[Theorem 3.3.4]{benson},  we have $$V\res{M}\simeq \dire{nMH}\langle nH\rangle_{RM},$$ with $\langle nH\rangle_{RM}\simeq R[M/M\cap H]$. It follows that $$(\langle nH\rangle_{RM})^M=\left(\sum_{m\in M/M\cap H}m\right)\langle nH\rangle_{RM},$$ thus $V^M=\left(\sum_{m\in M/M\cap H}m\right)V$. On the other hand,  if $u$ and $u'$ are elements of $V$ such that $u-u'\in I_MV$, then $$\sum_{m\in M/M\cap H}m(u-u') \in V^M\cap I_MV=0.$$ So we have  an $R[N/M]$-homomorphism defined by $\mathcal{N}_{M/M\cap H}:V_M\to V^M $ given by $$\mathcal{N}_{M/M\cap H}(u+I_MV)=\sum_{m\in M/M\cap H}mu.$$ This map is surjective and since $V_M$ is $R$-free, is  also injective, thus $V_M\simeq V^M$.
         \item Let $W=(V^M+ I_MV)/I_MV$. By the  discussion in part 2,  $W=p^kV_M$ where $p^k=[M:M\cap H]$.  By part $2$, we also have $V_M\simeq R[N/MH]$.  Consequently,  $$(R/p^kR)[N/MH]\simeq V_M/W\simeq (V/V^M)_M.$$ 
         \end{enumerate}
    \end{proof}
      
    \begin{lemma}\label{F(N) is invari}
    
         Let $U$ be an $RG$-lattice with $U\res{N}=\dire{H\in \Gamma}F(H)$.
         Then
         \begin{enumerate}
         \item The $\F N$-submodule  $\overline{F(N)}$ of $\overline{U_N}$ does not depend on the choice of decomposition and is $G$-invariant.
         \item The $\F G$-modules $\overline{U}_N/\overline{F(N)}$ and $\overline{(U/U^N)_N}$ are naturally isomorphic.
        \end{enumerate}
    \end{lemma}
         \begin{proof}
         We will proceed  as in \cite[Lemma 6]{John}.
             \begin{enumerate}
                 \item We have  $U^N=\dire{H\in \Gamma\setminus \{N\}}F(H)^N\oplus F(N),$
                 and $\varphi(F(H)^N)\leqslant \pi U_N$ if $H<N$, because $F(H)^N$ is generated by elements of the form $\sum_{g\in N/H}gu$ with $u\in F(H)$. So if $$\overline{\varphi}^N:U^N\to \overline{U_N}$$ is the obvious $RG$-homomorphism induced by $\varphi,$ then
                  $\overline{\varphi}^N(F(H)^N)=0$ for each $H<N$. Therefore, $\overline{\varphi}^N(U^N)=\overline{F(N)}$ is an $\F  G$-submodule of $\overline{U}_N$, and clearly it does not depend on the decomposition of $U\res{N}$.
                 \item We have a surjective map $\overline{\rho}:\overline{U_N}\longrightarrow \overline{(U/U^N)_N}$ induced by $\rho$, with $\overline{F(N)}\leqslant \ker \overline{\rho}$ since $F(N)\leqslant U^N$. Let $F:=\dire{\substack{H\in \Gamma\setminus \{N\}}}F(H)$. Then
                 $U/U^N=F/F^N$ and by Proposition \ref{tor}, the $R$-module $(F(H)/F(H)^N)_N$  is a direct sum of $R$-modules of the form $R/p^kR$  where $p^k=|N/H|$. It follows that $\overline{(U/U^N)_N}=\overline{(F/F^N)_N}$ is an $\F$-vector space with dimension equal to the  number $l$ of indecomposable   $RN$-direct summands of $F$. But $\overline{U_N}=\overline{F(N)}\oplus\overline{F_N}$ and $\dim(\overline{F_N})=l$ because $R[N/H]_N\simeq R$ as $RN$-modules, hence  $\overline{U_N}/\overline{F(N)}\simeq \overline{(U/U^N)_N}$.
             \end{enumerate}
        \end{proof}

    \begin{lemma}\label{t9}
         Let $N$ be a normal subgroup of $G$ and let $U$ be  an $RN$-lattice  with    maximal trivial   summand $T$. Suppose  that $U_N$ is an $R$-lattice  and $U_N=X\oplus Y$ is a decomposition into $RG$-modules  with $\overline{X}=\overline{T}$ in $\overline{U_N}$. Denote by $\theta_X$ the projection of $U_N$ on $X$ relative to the decomposition $U_N=X\oplus Y$. Then  the map $$\gamma:\xymatrix{U\ar[r]^{\varphi}& U_N\ar[r]^{\theta_X}& X },$$ is $RN$-split and $U=\ker\gamma\oplus T$.
    \end{lemma}
        \begin{proof}
            As $\overline{X}=\overline{T}$ and $T$ is a direct summand of $U_N$, we have $U_N=T\oplus Y$ because $U_N$ is a lattice. If $x\in X$ then we can express $x$ uniquely as $x=t+y$  with   $t\in T$ and $y\in Y$, thus we can   define $\alpha:X\longrightarrow U$ by $\alpha(x)=t$. Then, $\gamma\alpha(x)=\theta_X\varphi(\alpha(x))=\theta_X(\varphi(t))=\theta_X(x-y)=x$. Since $U_N=T\oplus Y=X\oplus Y$, each $0\neq t\in T\leqslant U_N$ is expressible uniquely as  $t=x+y$ with $x\neq 0$. It follows that $t=\alpha(x) $ and $\Ima \alpha=T$.
        \end{proof}
    \begin{lemma}\label{split-over-N}
        If the lattice $U$ satisfies:
            \begin{enumerate}
             \item $U_N$ is an $RG$-permutation module,
             \item $U\res{N}=\dire{H\in \Gamma}F(H)$
             and  $\overline{F(N)}$ has an $RG$ complement in $\overline{U_N}$,
        
            \end{enumerate}
        then, there exists an exact sequence of $RG$-lattices  
        $$\xymatrix{0\ar[r]& K\ar[r]& U\ar[r]& X\ar[r]& 0},$$
        such that $X$ and $K_N$ are $RG$-permutation modules, $X\downarrow_N\simeq F(N) $ and $U\downarrow_N\simeq K\oplus X$.
    \end{lemma}
        \begin{proof}
         By Lemma \ref{F(N) is invari}, $\overline{F(N)}$ is a $RG$-submodule of $\overline{U_N}$. By  Lemmas \ref{Ext=0} and \ref{permuExt=0}, there exists a $RG$-decomposition $U_N=X\oplus Y$ with $\overline{X}=\overline{F(N)}$.  From  Lemma \ref{t9}, we get an  exact sequence of $RG$-modules 
         $$ \xymatrix{0\ar[r] & K \ar[r] & U \ar[r]^{\gamma} & X \ar[r] & 0}$$
         that splits over $N$, with $K=\ker \gamma$. It follows that $U\downarrow_N=K\oplus F(N)$. So, we get  an exact sequence that splits over $G$: 
         $$\xymatrix{0\ar[r]&K_N\ar[r]&U_N\ar[r]^{\gamma_N}& X\ar[r]&0},$$ because $\gamma_N$ coincide with the projection of $U_N$ on $X$ relative to the decomposition $U_N=X\oplus Y$.
         Since $U_N$ is an $RG$-permutation module, it follows that  $K_N$ and $X$ are  $RG$-permutation modules.
    
        \end{proof}
        
    
    \begin{lemma}\label{G/N-coflasque}
        Let $U$ be an $RG$-lattice. Suppose that $N$ is a normal subgroup of $G$ such that: 
        \begin{enumerate}
            \item $U\res{N}=\dire{H\in \Gamma}F(H)$ with $F(N)=0$,
            \item $U_N$ is $G/N$-coflasque.    
        \end{enumerate}
        Then $U^N$ is $G/N$-coflasque.  
    \end{lemma}
    \begin{proof}
        For every $H\leqslant G$ such that $N\leqslant H$, the exact sequence
        $$\xymatrix{0\ar[r]& U^N\ar[r]^{\varphi^N}&U_N\ar[r]^-{\rho}&(U/U^N)_N\ar[r]&0}$$
        induces the exact sequence 
        $$\xymatrix{0\ar[r]& U^H\ar[r]^{\varphi^H}&(U_N)^H\ar[r]^-{\rho^H}&((U/U^N)_N)^H\ar[r]&\text{H}^1(H,U^N)\ar[r]&H^1(H,U_N)}.$$ Since $U_N$ is $G/N$-coflasque, we have $\text{H}^1(H,U_N)=0$.  So, we must show that the map $\rho_H: (U_N)^H\longrightarrow ((U/U^N)_N)^H$ is surjective. Since $F(N)=0$, we have $\overline{U_N}\simeq \overline{(U/U^N)_N}$ by Lemma \ref{F(N) is invari}. It follows that $\overline{(U)_N}^H\simeq \overline{(U/U^N)_N}^H$. Consider the following commutative diagram given by the obvious maps:
        $$\xymatrix{ (U_N)^H\ar[d]\ar[r]^-{\rho^H}& ((U/U^N)_N)^H\ar[d]\\\overline{U_N}^H\ar[r]^-{\overline{\rho}^H}&\overline{(U/U^N)_N}^H}$$ 
        Since the maps  $(U_N)^H\longrightarrow\overline{U_N}^H$  and $\overline{\rho}^H:\overline{(U_N)}^H\longrightarrow \overline{(U/U^N)_N}^H$ are surjective, the map
        $$((U/U^N)_N)^H\longrightarrow \overline{(U/U^N)_N}^H$$ is surjective. Then we have   
        $\overline{(U_N)^H}=\overline{U_N}^H\simeq \overline{(U/U^N)_N}^H=\overline{((U/U^N)_N)^H}$. Thus,  we have   $\overline{\rho^H}=\overline{\rho}^H$, and  by Nakayama's Lemma the map $\rho^H: (U_N)^H\longrightarrow ((U/U^N)_N)^H$ is surjective.  It follows that $\text{H}^1(H,U^N)=0$. By the inflation-restriction sequence \cite[Theorem 9.84]{Rotman},  we get $H^1(H/N,U^N)=H^1(H,U^N)$, thus  $U^N$ is $G/N $-coflasque.

    \end{proof}
    \begin{remark} Recall that an $RG$-permutation module is  $RG$-coflasque  \cite[Lemma 1.5]{Gregory}. However, the converse is not true in general.
    \end{remark} 
    \begin{prop}\label{block}
        Let $U$ be an $RG$-lattice and $N\lhd G$ with $|N|=p^l$. Let $\rho: U_N\to (U/U^N)_N$ be the natural projection and  suppose that $U\res{N}=\displaystyle{\bigoplus_{H\in \Gamma}F(H)}$. For any $k\geq 0$,  the $RN$-module  $\dire{\substack{H\in\Gamma\\|H|\geq p^k}}\overline{\rho(F(H)_N)}$ is an $RG$-submodule of $\overline{(U/U^N)_N}$. In particular, $\dire{\substack{H\in\Gamma\\|H|\geq p^k}}\overline{F(H)_N}$ is $G$-invariant in $\overline{U_N}$.
  \end{prop}
  \begin{proof}
     First, we observe that from the decomposition $U\res{N},$ we get  $$(U/U^N)_N=\dire{0\leq k\leq l-1}\dire{\substack{H\in \Gamma\\ |H|=p^k}} \rho(F(H)_N).$$ Also observe that   $\dire{\substack{H\in \Gamma\\ |H|=p^k}} \rho(F(H)_N)$ is an $R/p^{l-k}R$-free module by  Proposition \ref{tor}. Thus, $p^{l-k}$ acts as $0$ on $\dire{\substack{H\in \Gamma\\ |H|\geq p^k}} \rho(F(H)_N)$. Let $w\in\dire{\substack{H\in \Gamma\\ |H|\geq p^k}} \rho(F(H)_N)$, then $$gw=w_1+ w_2$$ where $$w_1\in \dire{\substack{H\in \Gamma\\ |H|\geq p^k}} \rho(F(H)_N) \;\;\text{and}\;\;w_2\in \dire{\substack{H\in \Gamma\\ |H|< p^k}} \rho(F(H)_N).$$ Since $p^{l-k}w=0$, we have $p^{l-k}gw=0$ and hence $p^{l-k}w_2=0$, because $p^{l-k}w_1=0$. It follows that $w_2\in \pi(U/U^N)_N$ since by Proposition \ref{tor}, $p^{l-k}$ does not act as $0$ on $\dire{\substack{H\in \Gamma\\ |H|< p^k}}\rho(F(H)_N)$. Therefore,  $\dire{\substack{H\in \Gamma\\ |H|\geq p^k}} \overline{\rho(F(H)_N)}$ is $G$-invariant in $\overline{(U/U^N)_N}$. By Lemma \ref{F(N) is invari}, we have a natural isomorphism $$\overline{(U/U^N)_N}\simeq \overline{U_N}/\overline{F(N)},$$ so the module  $\dire{\substack{H\in \Gamma\\ |H|\geq p^k}} \overline{F(H)_N}$ is an $RG$-submodule of $U_N$. 
  \end{proof}
    In the rest of the text we will make use of the notations introduced here. So suppose that  $U\res{N}=\bigoplus_{H\in \Gamma}F(H)$ and  let $M$ be a maximal subgroup of $N$ with $M\lhd G$.  By the Mackey decomposition formula \cite[Theorem 3.3.4]{benson}, we have $$R[N/H]\res{M}\simeq \dire{MgH}R[M/g(H\cap M)g^{-1}]\simeq \dire{MgH}R[M/H\cap M].$$ Thus, $F(H)\res{M}$ is a direct sum of modules of the form $R[M/H\cap M]$. Now, consider the restriction $(U\res{N})\res{M}=\dire{H\in \Gamma}F(H)\res{M}=\dire{H\in \Gamma'}F'(H)$ where $\Gamma'$ is a set of representatives of the distinct  orbits of the action by conjugation of $M$ on the set of  subgroups of $M$ and where $$F'(H)=\dire{\substack{K\in \Gamma\\K\cap M\in H^M}}F(K),$$
    where $H^M$ is the set of the $M$-conjugates of $H$.  
    
     With these notations, we have $$\dire{\substack{H\in\Gamma'\\ |H|\geq  p^k}}F'(H)=\dire{\substack{H\in \Gamma\\ |H\cap M|\geq p^k}}F(H).$$  The   Proposition \ref{block} applied to $M$ give us that the $RN$-submodule  $\dire{\substack{H\in\Gamma'\\ |H|\geq  p^k}}\overline{F'(H)_M}\leqslant \overline{U_M}$ is  an $RG$-mod\-ule. Hence, $$\dire{\substack{H\in\Gamma'\\ |H|\geq  p^k}}\overline{F'(H)_N}=\dire{\substack{H\in \Gamma\\ |H\cap M|\geq p^k}}\overline{F(H)_N}$$ is an $RG$-submodule of $\overline{U_N}$.   
   
    Let $M$ be a maximal subgroup of $N$ with $M\lhd G$. Then we can write $$(U_M)\res{N}=\dire{H\in \Gamma}F(H)_M$$ and relative to this decomposition, the maximal $R[N/M]$-trivial summand is $$F(N/M)=\dire{\substack{H\in \Gamma\\ H\nleqslant M}}F(H)_M,$$  because $(R[N/H])_M=R[N/MH]$, and $N=MH$ if and only if $H\nleqslant M$. 
    Since $(U_M)_N= U_N$ we have by Lemma \ref{F(N) is invari} that $\overline{F(N/M)_N}=\overline{F(N/M)}$ is an $RG$-submodule of $\overline{U_N}$.

  \begin{theorem}\label{M-block}
        Let $U$ be an $RG$-lattice and let  $M,N\lhd G$ with $M$ being a maximal subgroup of $N$. Suppose that $U\res{N}=\displaystyle{\bigoplus_{H\in \Gamma}F(H)}$ and  that $\dire{|H|\geq p^k}\overline{F(H)_N}$ is a direct summand of $\overline{U_N}$, for $k=1,...,|N|$.  Then $\dire{\substack{H\in \Gamma'\\ |H|\geq p^k}}\overline{F'(H)_N}$ is a direct summand of $\overline{U_N}$ as an $RG$-module.    
  \end{theorem}
  \begin{proof}
     Write $$\dbar{\substack{H\in \Gamma\\|H|\geq p^k}}{F(H)_N}=\dbn{\substack{H\in \Gamma\\ |H\cap M|\geq p^k}}{H}\;\oplus \dbn{\substack{H\in \Gamma\\ |H\cap M|<p^k\\ |H|\geq p^k}}{H}.$$ Now, observe that $$\dbn{\substack{H\in \Gamma\\ |H\cap M|<p^k\\ |H|\geq p^k}}{H}=\dbn{\substack{H\in \Gamma\\  H\nleqslant M\\|H|=p^k}}{H},$$ because being $M$ maximal we have $|H\cap M|> p^{k-1}$ if $|H|> p^k$. Since $$F(N/M)=\dire{\substack{H\in \Gamma\\ H\nleqslant M\\ |H|\geq p}}F(H),$$ we get 
     $$\overline{F(N/M)}\cap \dbn{\substack{H\in \Gamma\\ |H|\geq p^k}}{H}=\dbn{\substack{H\in \Gamma\\ |H\cap M|<p^k\\ |H|= p^k}}{H}\;\oplus\dbn{\substack{H\in \Gamma\\ H\nleqslant M\\ |H\cap M|\geq p^k}}{H}.$$
     By Lemma \ref{F(N) is invari} and Proposition \ref{block}, we have  $\overline{F(N/M)}$  and $\dbn{\substack{H\in \Gamma\\ |H|\geq p^k}}{H}$ $RG$-modules, thus the  intersection of these $RG$-modules is an $RG$-module. 
     It follows that $$Z_k:=\dbn{\substack{H\in \Gamma\\ |H\cap M|<p^k\\ |H|=p^k}}{H}\;\oplus\dbn{\substack{H\in \Gamma\\|H|\geq p^{k+1}}}{H}=\overline{F(N/M)}\cap \dbn{\substack{H\in \Gamma\\ |H|\geq p^k}}{H} + \dbn{\substack{H\in \Gamma\\|H|\geq p^{k+1}}}{H},$$
     is an $RG$-module by  Proposition \ref{block}. Since $\dbn{|H|\geq p^{k+1}}{H}$ is an $RG$-direct summand of $\overline{U_N}$, there exists an $RG$-complement $W_{k}\leqslant Z_k$ such that $$Z_k=W_k\oplus \dbn{|H|\geq p^{k+1}}{H}.$$ Note that   $Z_k\cap \left(\dbn{\substack{H\in \Gamma\\ |H\cap M|\geq p^{k}}}{H}\right)=\dbn{\substack{H\in \Gamma\\|H|\geq p^{k+1}}}{H}$, from which   we get $$\dbn{|H|\geq p^{k}}{H}=\dbn{\substack{H\in \Gamma\\|H\cap M|\geq p^k}}{H}\oplus W_k.$$  By  Proposition \ref{block}, we have $\dire{\substack{H\in \Gamma'\\ |H|\geq  p^k}}\overline{F'(H)_N}=\dbn{\substack{H\in \Gamma\\ |H\cap M|\geq p^k}}{H}$ an $RG$-module, therefore an $RG$-direct summand of $\dbn{|H|\geq p^k}{H}$, and hence of $\overline{U_N}$.
     \end{proof}
    
    \begin{prop}\label{sumanofU_N}
         Let $U$ be an $RG$-lattice and $N\lhd G$ with $|N|=p^l$. Suppose that $U\res{N}=\displaystyle{\bigoplus_{H\in \Gamma}F(H)}$ with  $\dire{|H|\geq p^k}\overline{F(H)_N}$ being a  direct summand of $\overline{U_N}$, for $k=1,...,l$. If $M$ is a maximal subgroup of $N$ with $M\lhd G$, then 
        \begin{enumerate}
        \item \label{1sum}  $\overline{F(N/M)}$ is an $RG$-direct summand of $\overline{U_N}$;
        \item $\overline{F(N/M)}\cap \left(\dire{\substack{H\in \Gamma'\\ |H|\geq p^k}}\overline{F'(H)_N}\right)$ is an $RG$-direct summand of $\dire{\substack{H\in \Gamma'\\ H\geq p^k}}\overline{F'(H)_N}$, for $k=1,...,l-1$;
        \item $\dire{\substack{H\in \Gamma\\ H\nleqslant M\\ |H|<p^{k+1}}}\overline{F(H)_N}\;\oplus\dire{\substack{H\in \Gamma'\\|H|\geq p^k}}\overline{F'(H)_N}$ is an $RG$-direct summand of $\overline{U_N}$, for $k=1,...,l-1$. 
        \end{enumerate}
    \end{prop}
     \begin{proof}
       \begin{enumerate}   
        \item First, observe that $$\dbn{\substack{H\in \Gamma\\|H|\geq p}}{H}=\dbn{\substack{H\in \Gamma\\ H\nleqslant M\\|H|\geq p}}{H}\oplus\dbn{\substack{H\in \Gamma\\ H\leqslant M\\|H|\geq p}}{H},$$ and $\overline{F(N/M)}=\dbn{\substack{H\in \Gamma\\ H\nleqslant M\\|H|\geq p}}{H}.$  
      Proposition \ref{block} applied to $M$ and $N$ gives us that  $$\dire{\substack{H\in \Gamma'\\|H|\geq p}}\overline{F'(H)_N}=\dire{\substack{H\in \Gamma\\|H\cap M|\geq p}}\overline{F(H)_N}=\dbn{\substack{H\in \Gamma\\ H\leqslant M\\|H|=p}}{H}\oplus\dbn{\substack{H\in \Gamma\\ |H\cap M|\geq p\\|H|\geq p^2}}{H}$$ and 
      $$Z_1:=\dbn{\substack{H\in \Gamma'\\|H|\geq p}}{H}+ \dbn{\substack{H\in \Gamma\\|H|\geq p^2}}{H}=\dbn{\substack{H\in \Gamma\\ H\leqslant M\\|H|=p}}{H}\;\oplus\dbn{\substack{H\in \Gamma\\|H|\geq p^2}}{H}$$
      are $RG$-modules. Since $\dbn{|H|\geq p^2}{H}$ is an $RG$-direct summand of $\overline{U_N}$, then there exists an $RG$-complement $W_{1}\leqslant Z_1$ such that $$Z_1=W_1\oplus \dbn{|H|\geq p^{2}}{H}.$$ Since $Z_1\cap \left(\dbn{\substack{H\in \Gamma\\ H\nleqslant M\\|H|\geq p}}{H}\;\oplus\dbn{\substack{H\in \Gamma\\H\leqslant M\\|H|\geq p^2}}{H}\right)=\dbn{|H|\geq p^{2}}{H}$, we get 
      $$\dbn{|H|\geq p}{H}=\dbn{\substack{H\in \Gamma\\ H\nleqslant M\\|H|\geq p}}{H}\;\oplus\dbn{\substack{H\in \Gamma\\H\leqslant M\\|H|\geq p^2}}{H}\oplus W_1.$$ 
     By similar arguments, we get  $RG$-submodules $ W_2\leqslant Z_2,...,W_{l-1}\leqslant Z_{l-1}$, where $$Z_{k}= \dbn{\substack{ H\in \Gamma\\H\leqslant M\\|H|=p^{k}}}{H}\;\oplus\dbn{\substack{H\in \Gamma\\ |H|\geq p^{k+1}}}{H}= \dbn{\substack{H\in \Gamma\\|H|\geq p^{k+1}}}{H}\oplus W_k\leqslant \dbn{|H|\geq p^{k}}{H},$$
      such that $$\dbn{\substack{H\in \Gamma\\|H|\geq p}}{H}=\dbn{\substack{H\in \Gamma\\ H\nleqslant M\\|H|\geq p}}{H}\oplus W_{l-1}\oplus W_{l-2}\oplus ...\oplus W_1.$$ 
        It follows that $$\overline{F(N/M)}=\dbn{\substack{H\in \Gamma\\ H\nleqslant M\\|H|\geq p}}{H}$$ is an $RG$-direct summand of $\dbn{\substack{H\in \Gamma\\|H|\geq p}}{H}$, and hence of $\overline{U_N}$. 

        \item Clearly, the result is true for $k=l-1$ and $k=l$. So suppose $1\leq k\leq l-2$.  We have $$\dire{\substack{H\in\Gamma'\\|H|\geq p^k}}\overline{F'(H)_N}=\dire{\substack{H\in \Gamma\\ |H\cap M|\geq p^k}}\overline{F(H)_N}=\dbn{\substack{H\in \Gamma\\ H\nleqslant M\\|H\cap M|\geq p^k}}{H}\;\oplus\dbn{\substack{H\in \Gamma\\ H\leqslant M\\|H|\geq p^k}}{H}$$ and $$\overline{F(N/M)}\cap \dire{\substack{H\in \Gamma\\ |H\cap M|\geq p^k}}\overline{F(H)_N}=\dbn{\substack{H\in \Gamma\\ H\nleqslant M\\|H\cap M|\geq p^k}}{H}.$$
        We will show that the above intersection is an $RG$-direct summand of $\dbn{\substack{H\in \Gamma\\ |H|\geq p^{k+1}}}{H}$ which is a direct summand of $\overline{U_N}$. For this, observe that  $$\dbn{\substack{H\in \Gamma\\ |H|\geq p^{k+1}}}{H}=\dbn{\substack{H\in \Gamma\\ H\nleqslant M\\|H\cap M|\geq p^k}}{H}\;\oplus\dbn{\substack{H\in \Gamma\\ H\leqslant M\\|H|\geq p^{k+1}}}{H}.$$ By Proposition \ref{block}, the $RN$-submodule   $$ \dire{\substack{H\in \Gamma'\\ |H|\geq p^{k+1}}}\overline{F'(H)_N}=\dbn{\substack{H\in \Gamma\\ |H\cap M|\geq p^{k+1}}}{H}=\dbn{\substack{H\in \Gamma\\ H\leqslant M\\ |H|=p^{k+1}}}{H}\;\oplus\dbn{\substack{H\in \Gamma\\ |H\cap M|\geq p^{k+1}\\|H|\geq p^{k+2}}}{H},$$
        is an $RG$-submodule, hence $$Z_k:=\dire{\substack{H\in \Gamma'\\ |H|\geq p^{k+1}}}\overline{F'(H)_N}+ \dbn{\substack{H\in \Gamma\\|H|\geq p^{k+2}}}{H}=\dbn{\substack{H\in \Gamma\\ H\leqslant M\\|H|=p^{k+1}}}{H}\;\oplus\dbn{\substack{H\in \Gamma\\|H|\geq p^{k+2}}}{H}$$ is an $RG$-submodule of $\dbn{\substack{H\in \Gamma\\ |H|\geq p^{k+1}}}{H}$. By  hypothesis, we have that $\dbn{\substack{H\in \Gamma\\ |H|\geq p^{k+2}}}{H}$ is an $RG$-direct summand of $\overline{U_N}$, it follows that we can find an $RG$-complement $W_{k}\leqslant Z_k$ such that $$Z_k=W_k\oplus \dbn{\substack{H\in \Gamma\\ |H|\geq p^{k+2}}}{H}.$$ Since $Z_k\cap \left(\dbn{\substack{H\in \Gamma\\ H\nleqslant M\\|H\cap M|\geq p^k}}{H}\;\oplus\dbn{\substack{H\in \Gamma\\ H\leqslant M\\|H|\geq p^{k+2}}}{H}\right)=\dbn{\substack{H\in \Gamma\\ |H|\geq p^{k+2}}}{H}$, we get 
        $$\dbn{\substack{H\in \Gamma\\ |H|\geq p^{k+1}}}{H}=\dbn{\substack{H\in \Gamma\\ H\nleqslant M\\|H\cap M|\geq p^k}}{H}\;\oplus\dbn{\substack{H\in \Gamma\\ H\leqslant M\\|H|\geq p^{k+2}}}{H}\oplus W_k.$$ Continuing in this way, we find an $RG$-submodule $W\leqslant \dbn{\substack{H\in \Gamma\\ |H|\geq p^{k+1}}}{H}$ such that $$\dbn{\substack{H\in \Gamma\\ |H|\geq p^{k+1}}}{H}=\dbn{\substack{H\in \Gamma\\ H\nleqslant M\\|H\cap M|\geq p^k}}{H}\oplus W.$$ By hypothesis, $\dbn{\substack{H\in \Gamma\\ |H|\geq p^{k+1}}}{H}$ is an $RG$-direct summand of $\overline{U_N}$.  Therefore, $\dbn{\substack{H\in \Gamma\\ H\nleqslant M\\|H\cap M|\geq p^k}}{H}$ is an $RG$-direct summand of $\dire{\substack{H\in \Gamma'\\ |H|\geq p^k}}\overline{F'(H)_N}$.
        \item 
         Note that $$\overline{F(N/M)}+ \dire{\substack{H\in \Gamma' \\|H|\geq p^{k} }}\overline{F'(H)_N}=\dbn{\substack{H\in \Gamma\\ H\nleqslant M\\|H|<p^{k+1}}}{H}\;\oplus\dire{\substack{H\in \Gamma' \\|H|\geq p^{k} }}\overline{F'(H)_N}$$ is an $RG$-module and  $$\overline{U_N}=\dbn{\substack{H\in \Gamma\\ H\nleqslant M\\|H|<p^{k+1}}}{H}\;\oplus\dire{\substack{H\in \Gamma' \\|H|\geq p^{k} }}\overline{F'(H)_N}\;\oplus\dbn{\substack{H\in \Gamma\\ H\leqslant M\\ |H|<p^k}}{H}.$$ Since $$\dire{\substack{H\in \Gamma'\\ |H|\geq p}}\overline{F'(H)_N}=\dbn{\substack{H\in \Gamma\\ H\leqslant M\\ |H|=p}}{H}\;\oplus\dbn{\substack{H\in \Gamma\\|H\cap M|\geq p\\|H|\geq p^2}}{H}$$ is an $RG$-module, then   $$Z_1=\dire{\substack{H\in \Gamma'\\ |H|\geq p}}\overline{F'(H)_N}+\dbn{\substack{H\in \Gamma\\|H|\geq p^2}}{H}=\dbn{\substack{H\in \Gamma\\ H\leqslant M\\ |H|=p}}{H}\;\oplus\dbn{\substack{H\in \Gamma\\|H|\geq p^2}}{H}$$ is an $RG$-module. Since $\dbn{\substack{H\in \Gamma\\|H|\geq p^2}} {H}$ is an $RG$-direct summand of $\overline{U_N}$, we can find an $RG$-complement $W_1\leqslant Z_1,$ such that $$Z_1=W_1\oplus \dbn{\substack{H\in \Gamma\\|H|\geq p^2}}{H}.$$ As $$Z_1\cap \left(\dbn{\substack{H\in \Gamma\\ H\nleqslant M\\|H|<p^{k+1}}}{H}\;\oplus\dbn{\substack{H\in \Gamma \\|H\cap M|\geq p^{k} }}{H}\;\oplus\dbn{\substack{H\in \Gamma\\ H\leqslant M\\ p^2\leq|H|<p^K}}{H}\right)=\dbn{\substack{H\in \Gamma\\|H|\geq p^2}}{H},$$ we get
        $$\overline{U_N}=\dbn{\substack{H\in \Gamma\\ H\nleqslant M\\|H|<p^{k+1}}}{H}\;\oplus\dbn{\substack{H\in \Gamma \\|H\cap M|\geq p^{k} }}{H}\;\oplus\dbn{\substack{H\in \Gamma\\ H\leqslant M\\ p^2\leq|H|<p^K}}{H}\oplus W_1.$$
        Continuing in this way, we can find $RG$-submodules $W_1, W_2,...,W_{k-1}$ such that $$\overline{U_N}=\dbn{\substack{H\in \Gamma\\ H\nleqslant M\\|H|<p^{k+1}}}{H}\;\oplus\dire{\substack{H\in \Gamma' \\|H|\geq p^{k} }}\overline{F'(H)_N}\oplus W_{k-1}\oplus W_{k-2}\oplus ...\oplus W_1.$$ It follows that $\dbn{\substack{H\in \Gamma\\ H\nleqslant M\\|H|<p^{k+1}}}{H}\;\oplus\dbn{\substack{H\in \Gamma\\|H\cap M|\geq p^k}}{H}$ is an $RG$-direct summand of $\overline{U_N}$.
        \end{enumerate}

  \end{proof}
  
 We say that a finitely generated $RG$-module $V$ is an $RG$-\textit{block module} if $$V=\dire{i\in I}B_i,$$ where $B_i$ is an $A_iG$-lattice with $A_i=R/p^{i}R$  and $I\subseteq \{1,...,l\}$, for some $l\geq 0$. If every $B_i$ is an $A_iG$-permutation module, then we say that $V$ is an $RG$-\textit{permutation block module}.   
\begin{prop}\label{block-module}
    Let $U$ be an $RG$-lattice and $N\lhd G$ with $|N|=p^l$. Suppose that $U\res{N}=\dire{H\in \Gamma}F(H)$ and   $(U/U^N)_N$ is an $RG$-block module. If $F(N)=0$, then  $\dbn{|H|\geq p^k}{H}$ is an $RG$-direct summand of $\overline{U_N}$ for each  $k=1,...,l-1$. 
\end{prop}
\begin{proof}
    We can write $(U/U^N)_N=\displaystyle{\bigoplus_{i=1}^{|N|}B_{i}}$ where $B_{i}$ is an $(R/p^{i}R)[G/N]$-lattice, because, by  Proposition \ref{tor}, the decomposition of $U\res{N}$ implies that the indecomposable $R$-summands of $(U/U^N)_N$ are of the form $R/p^iR$ for some $i=1,2,...,l$. 
    Now, from the decomposition of $(U/U^N)_N$ as a direct sum of  $B_i's$, we get   $$\dire{\substack{H\in \Gamma\\|H|\geq p^{i}}}\rho(F(H)_N)\leqslant\bigoplus_{j\geq i} B_{l-j} + \pi(U/U^N)_N,$$ because $p^{l-i}$ acts as $0$ on $\dire{\substack{H\in \Gamma\\|H|\geq p^{i}}}\rho(F(H)_N)$. It follows that $$\dire{\substack{H\in\Gamma\\|H|\geq p^{i}}}\overline{\rho(F(H)_N)}\leqslant\bigoplus_{j\geq i} \overline{B_{l-j}} \leqslant \overline{(U/U^N)_N}.$$  Since  $\dire{\substack{H\in \Gamma\\|H|\geq p^{i}}}\rho(F(H)_N)$ is an $R$-direct summand of $(U/U^N)_N$, by similar arguments we have 
    $$\bigoplus_{j\geq i} \overline{B_{l-j}}\leqslant\dire{\substack{H\in\Gamma\\|H|\geq p^{i}}}\overline{\rho(F(H)_N)}.$$
    Therefore,  $$\dire{\substack{H\in\Gamma\\|H|\geq p^{i}}}\overline{\rho(F(H)_N)}=\bigoplus_{j\geq i} \overline{B_{l-j}}.$$ Finally, if $F(N)=0$ then the   Lemma \ref{F(N) is invari} provides us a natural isomorphism  $  \overline{(U/U^N)}\simeq \overline{U_N}$, thus $$\bigoplus_{j\geq i} \overline{B_{l-j}}=\dire{\substack{H\in\Gamma\\|H|\geq p^{i}}}\overline{\rho(F(H)_N)}=\dbn{\substack{H\in\Gamma\\|H|\geq p^i}}{H},$$ is an $RG$-direct summand of $\overline{U_N}$.

\end{proof}

\begin{prop}(A version of \cite[Proposition 3.5]{MACQUARRIE2020106925} )\label{exchange}
Let  $X$ be a  direct summand of a finitely generated $\F G$-module

$$
Y \;=\; \dire{i\in I} Y_i,$$

where each $Y_i$ is an indecomposable $\F G$-module.
Then there exists a subset $J \subseteq I$ such that

$$
Y \;=\; X \oplus\dire{i\in J} Y_i.$$

\end{prop}

\begin{prop}\label{sumanofU_M}
    Let $U$ be an $RG$-lattice and $N\lhd G$ with $|N|=p^l$. Suppose that $U\res{N}=\displaystyle{\bigoplus_{H\in \Gamma}F(H)}$ and that  $\dire{|H|\geq p^k}\overline{F(H)_N}$ is an  $RG$-direct summand of $\overline{U_N}$, for $k=1,...,l$. If $M$ is a maximal subgroup of $N$ with $M\lhd G$ and $\overline{U_M}$ is an $\F G$-permutation module, then $\dire{\substack{H\in \Gamma'\\ |H|\geq p^k}}\overline{F'(H)_M}$ is an $RG$-direct summand of $\overline{U_M}$, for $k=1,...,l-1$.
\end{prop}
\begin{proof}
    Write $\overline{U_M}=\dire{i\in I}P_i$ where $P_i$ is an indecomposable $\F G$-permutation module. Being $(P_i)_N$ an indecomposable $\F G$-module and $\dbn{\substack{H\in \Gamma\\ |H|\geq p^k}}{H}$ is an $RG$-direct summand of $\overline{U_N}$,  
     the  Proposition \ref{exchange}  give us  $J\subseteq I$ such that $$\overline{U_N}=\dire{j\in J}(P_j)_N\oplus \dire{\substack{H\in \Gamma\\ |H|\geq p^k}}\overline{F(H)_N},$$ because $\overline{U_M}=\overline{U}_M$ and $\overline{U}_N=\overline{U_N}$.  Since $\overline{U_M}$ is an $\F G$-permutation module, then $\overline{U_M}\res{N/M}$ has an maximal trivial summand $X$ such that $X$ is an $\F G$-module. By Lemma \ref{F(N) is invari}, the image of  $X$ and   $F(N/M)$ coincide  in $\overline{U_N}$. So there exists $ I_1\subseteq I$ such that  $\dire{i\in I_1}P_i$ is the maximal $RN$-trivial summand of $\overline{U_M}$ and $$\overline{F(N/M)_N}=\dire{i\in I_1}(P_i)_N.$$ So, by  Proposition \ref{sumanofU_N}, there exists $I_2\subseteq I_1$ such that  $$\overline{F(N/M)_N}=\dire{i\in I_2}(P_i)_N\oplus\left(\overline{F(N/M)_N}\cap\left(\dire{\substack{H\in \Gamma'\\ |H|\geq p^k}}\overline{F'(H)_N}\right)\right).$$ Then by   Proposition \ref{sumanofU_N},  the module  $\dire{i\in I_2}(P_I)_N\;\oplus\dire{\substack{H\in \Gamma' \\ |H|\geq p^k}}\overline{F'(H)_N}$ is an $RG$-direct summand of $\overline{U_N}$. So there exists $I_3\subseteq I$ such that $$\overline{U_N}=\dire{i\in I_3}(P_i)_N\;\oplus\dire{i\in I_2}(P_i)_N\;\oplus\dire{\substack{H\in \Gamma'\\ |H|\geq p^k}}\overline{F'(H)_N}.$$ Consequently, by Nakayama's Lemma, we can write 
    $$\overline{U_M}=\dire{i\in I_3}P_i \;\oplus\dire{i\in I_2}P_i+\dire{\substack{H\in \Gamma'\\ |H|\geq p^k}}\overline{F'(H)_M}.$$ We will show that this sum is  direct.
    By Proposition \ref{sumanofU_N}, the $RG$-module  $$Z:=\overline{F(N/M)_N}\cap \left(\dire{\substack{H\in \Gamma'\\ |H|\geq p^k}}\overline{F'(H)_N}\right)$$ is an $RG$-direct summand of $\dire{\substack{H\in \Gamma'\\ |H|\geq p^k}}\overline{F'(H)_N}$, so there exists $W\leqslant \dire{\substack{H\in \Gamma'\\ |H|\geq p^k}}\overline{F'(H)_M}$ such that 
    $$\overline{U_N}=\dire{i\in I_3}\overline{(P_i)_N}\;\oplus \dire{i\in I_2}\overline{(P_i)_N}\oplus Z\oplus W.$$
     Let $\widehat{N/M}=\sum_{g\in N/M}g$ and define $\mathcal{N}_{N/M}:U_N\to U_M$ by 
    $\mathcal{N}_{N/M}(u+I_NU_M)=\widehat{N/M}u$ for $u\in U_M$. This map induces an $RG$-homomorphism $\overline{\mathcal{N}}_{N/M}:\overline{U_N}\to \overline{U_M}$ and by above decomposition of $\overline{U_N}$, we have $$\widehat{N/M}(\overline{U_M})=\overline{\mathcal{N}_M}(\overline{U_N})=\overline{\mathcal{N}_M}(\dire{i\in I_3}\overline{(P_i)_N})\oplus \overline{\mathcal{N}_M}(W),$$ because $$\ker(\overline{\mathcal{N}_M})=\overline{F(N/M)}=\dire{i\in I_2}\overline{(P_i)_N}\oplus Z.$$
    This shows that $$\widehat{N/M}(\overline{U_M})=\widehat{N/M}\left(\dire{i\in I_3}P_i\right)\oplus \widehat{N/M}\left(\dire{\substack{H\in \Gamma'\\ |H|\geq p^k}}\overline{F'(H)_M}\right).$$
    From the decomposition of $$\overline{U_N}=\dire{i\in I_3}(P_i)_N\;\oplus\dire{i\in I_2}(P_i)_N\;\oplus\dire{\substack{H\in \Gamma'\\ |H|\geq p^k}}\overline{F'(H)_N},$$
     we get      $$\left(\dire{i\in I_3}P_i\;\oplus\dire{i\in I_2}P_i\right)\cap\left(\dire{\substack{H\in \Gamma'\\ |H|\geq p^k}}\overline{F'(H)_M}\right)\leqslant I_N\overline{U_M},$$ but $\dire{i\in I_3}P_i\;\oplus\dire{i\in I_2}P_i$ and  $\dire{\substack{H\in \Gamma'\\ |H|\geq p^k}}\overline{F'(H)_M}$ both being $RN$-direct summands of $\overline{U_M}$, we have $$\left(\dire{i\in I_3}P_i\;\oplus\dire{i\in I_2}P_i\right)\cap\left(\dire{\substack{H\in \Gamma'\\ |H|\geq p^k}}\overline{F'(H)_M}\right)\leqslant I_N\left(\dire{i\in I_3}P_i\;\oplus\dire{i\in I_2}P_i\right)\cap I_N\left(\dire{\substack{H\in \Gamma'\\ |H|\geq p^k}}\overline{F'(H)_M}\right)= $$
    $$=I_N\left(\dire{i\in I_3}P_i\right)\cap I_N\left(\dire{\substack{H\in \Gamma'\\ |H|\geq p^k}}\overline{F'(H)_M}\right).$$
    If $u$ is an element in the above intersection and  if $n$ generates $N/M$, there exists $r\geq 0$ such that $(n-1)^ru=0$, hence $$(n-1)^{r-1}u\in \left(\dire{i\in I_3}I_NP_i\right)^N\cap\left(\dire{\substack{H\in \Gamma'\\|H|\geq p^k}}I_N\overline{F'(H)_M}\right)^N.$$ But $\dire{i\in I_3}P_i^N=\widehat{N/M}\left(\dire{i\in I_3}P_i\right)$ and $\dire{\substack{H\in \Gamma'\\|H|\geq p^k}}\left(I_N\overline{F'(H)_M}\right)^N\leqslant \widehat{N/M}\left(\dire{\substack{H\in \Gamma'\\ |H|\geq p^k}}\overline{F'(H)_M}\right) $, since  $$\widehat{N/M}\overline{U_M}=\widehat{N/M}\left(\dire{i\in I_3}P_i\right)\oplus \widehat{N/M}\left(\dire{\substack{H\in \Gamma'\\ |H|\geq p^k}}\overline{F'(H)_M}\right),$$ so $(n-1)^{r-1}u=0$. Continuing in this way, we get $u=0$.  
    It follows that $$\overline{U_M}=\dire{i\in I_3}P_i\;\oplus \dire{i\in I_2}P_i\;\oplus\dire{\substack{H\in \Gamma'\\ |H|\geq p^k}}\overline{F'(H)_M}.$$
\end{proof}
Let $U$ and $V$ be $RG$-modules. An $RG$-homomorphism  $f:U\longrightarrow V$ is called  supersurjective  if the induced map $f^H:U^H\longrightarrow V^H$ is surjective for every subgroup $H$ of $G$.
    \begin{prop}(\cite[Lemma 6.6]{MACQUARRIE2020106925})\label{super}
        Let $A$ denote either $R$ or $\F$. Let   $U$ and $V$ be $AG$-lattices and suppose that $V$ is an $AG$-permutation module. If $f:U\longrightarrow V$ is a supersurjective $AG$-module homomorphism, then $f$ splits. 
      \end{prop}
 
\section{Proof of main Theorem}\label{s3}
    \begin{proof}[Proof of Theorem \ref{main}]
        We argue by induction on $\text{rank}(U)$. The result is obvious if either $\text{rank}(U)=1$ or $|G|=p$. So suppose $|G|>p$ and $\text{rank}(U)>1$. We split the proof into two cases: 1 when $F(N)=0$ and 2 when $F(N)\neq 0$. 
        \begin{enumerate}
        \item We first consider the case where $F(N)=0$.  The result is true if $|N|=p$ by Theorem \ref{Theo.MSZ}, because in this case $U\res{N}$ is $RN$-free, thus  $U^N\simeq U_N$. So suppose $|N|>p$ and let $M$ be a maximal subgroup of $N$ with $M\lhd G$.  Since  $$F(N/M)=\dire{\substack{H\in \Gamma\\H\nleqslant M}}F(H)_M$$  is the $RN$-maximal trivial summand of $(U_M)\res{N}$ relative to the decomposition $$(U_M)\res{N}=\dire{\substack{H\in \Gamma}} F(H)_M,$$ then $\overline{F(N/M)_N}$ is an $RG$-direct summand of $\overline{U_N}$ by  Propositions \ref{block-module} and  \ref{sumanofU_N} part \ref{1sum}. So, by Lemma \ref{split-over-N} there exists an exact sequence of $RG$-modules $$\xymatrix{0\ar[r] & K'\ar[r] & U\ar[r]& X' \ar[r]& 0}$$ that splits over $N$, where $X'$ is an $RG$-permutation module such that $X\res{N}\simeq F(N/M)$ and $K'_N$ is an $RG$-permutation module. 
         Since $K'\res{N/M}$ is $R[N/M]$-free and $K'_N$ is an $RG$-permutation module, by  Theorem \ref{Theo.MSZ}, we have $K'$ an  $R[G/M]$-permutation module. So by Lemma \ref{permuExt=0}, the above exact sequence splits and we get $U\simeq K'\oplus X'$ an $RG$-permutation module.
          Now, consider the decomposition $$U\res{M}=\dire{H\in \Gamma}F(H)\res{M}=\dire{H\in  \Gamma'}F'(H).$$
        If $F'(M)\neq 0$, then by Lemma \ref{F(N) is invari} the $RN$-module  $\overline{F'(M)}$  is an $RG$-submodule of  $\overline{U_M}$.  By  Propositions \ref{block-module} and  \ref{sumanofU_M}, we have $\dire{\substack{H\in \Gamma'\\ |H|\geq p^i}}\overline{F'(H)_M}$ is an $RG$-direct summand of $\overline{U_M}$ for $i=1,...,l-1$. Therefore, $\overline{F'(M)}$ is an $RG$-direct summand of $\overline{U_M}$. Since $U_M$ is an $R[G/M]$-permutation module,  by Lemma \ref{Ext=0} and Lemma \ref{permuExt=0}, we get an $RG$-decomposition $$U_M=Y\oplus X$$ such that $\overline{X}= \overline{F'(M)}$. By Lemma \ref{split-over-N} this decomposition gives us an exact sequence $$\xymatrix{0\ar[r] & K\ar[r]& U\ar[r]& X\ar[r]& 0}$$ that splits over $M$, furthermore,  $K\res{M}$ does not have trivial summands and $K_M$ is an $RG$-permuta\-tion module. So  Proposition \ref{G/N-coflasque} ensures that $K^M$ is $G/M$-coflasque. Since $X$ is an $RG$-permutation module,  the above exact sequence induces the exact sequence  $$\xymatrix{0\ar[r] & K^M\ar[r]& U^M\ar[r]& X\ar[r]& 0}$$ that splits over $G$, because $K^M$ being a $G/M$-coflasque module implies that the map $U^M\to X$ is supersurjective and so split by Proposition \ref{super}. Therefore, the exact sequence $$\xymatrix{0\ar[r] & K\ar[r]& U\ar[r]& X\ar[r]& 0}$$ splits over $G$ because $M$ acts trivially on $X$. It follows that $U\simeq K\oplus X$ and $(U/U^N)_N\simeq(K/K^N)_N\oplus (X/X^N)_N$. Since $(U/U^N)_N$ is an $R[G/N]$-permutation block module, then $(K/K^N)_N$ is an $RG$-permutation block module. Since $K^N$ is $G/N$-coflasque and $K_N$ is a permutation module, by the induction hypothesis on $\text{rank}(U)$ we have $K$ an $RG$-permu\-tation module. Thus, $U$ is an $RG$-permu\-tation module.
        Remains to consider the case  when 
        $F'(M)=0$. In this case, we can find a maximal subgroup $M_1$ of $M$ that is normal in $G$, and as we show that the modules $\dire{\substack{H\in \Gamma'\\|H|\geq p^k}}\overline{F'(H)_M}$ are $RG$-direct summands of $\overline{U_M}$ for $k=1,...,l-1$, we can 
        repeat the same arguments as in the case $F(N)=0$ to show that $U$ is an $RG$-permutation module if $U\res{M_1}$ has trivial summands. If $M_1$ has no trivial summands, we can continue this process with a maximal subgroup $M_2$ of $M_1$ that is normal in $G$. If $U\res{M_2}$ has trivial summands, we can use the same arguments as above to show that $U$ is an $RG$-permutation module. If not, we can continue this process.  Of course, if during this process we find  a normal subgroup $M_i$ with $|M_i|=p$, then if $U\res{M_i}$ does not have a trivial summand, we can use  Theorem \ref{Theo.MSZ} to shown that $U$ is an $RG$-permutation module, because during this process we have shown that $U_{M_i}$ is an $R[G/M_i]$-permutation module.
        \item
        For the case $F(N)\neq 0$, consider the exact sequence $$\xymatrix{0\ar[r] & \overline{F(N)}\ar[r]& \overline{U_N}\ar[r]& \overline{(U/U^N)_N}\ar[r]& 0}$$
        induced by $\rho:U_N\to (U/U^N)_N$.
        Since $U^N$ is $G/N$-coflasque, $U_N$ is an $RG$-permutation module  and $(U/U^N)_N$ is an $RG$-permutation block module, in the following commutative diagram 
        $$\xymatrix{(U_N)^H\ar[r]\ar[d]& ((U/U^N)_N)^H\ar[d]\\ \overline{U_N}^H\ar[r]& \overline{(U/U^N)_N}^H}$$ the top map is surjective and the vertical maps are surjective. So, the bottom map is surjective. It follows that the map $$\overline{U_N}\to \overline{(U/U^N)_N}$$ is surpersurjective. As consequence, the exact sequence $$\xymatrix{0\ar[r] & \overline{F(N)}\ar[r]& \overline{U_N}\ar[r]& \overline{(U/U^N)_N}\ar[r]& 0}$$ splits over $G$, because $\overline{(U/U^N)_N}$ is an $\F G$-permutation module. Thus, by Lemma \ref{split-over-N} we get an exact sequence $$\xymatrix{0\ar[r] & K\ar[r]& U\ar[r]& X\ar[r]& 0}$$ that splits over $N$. Its follows that $(U/U^N)_N=(K/K^N)_N$ and $U_N=K_N\oplus X$. In particular, by  Proposition \ref{G/N-coflasque}, $K^N$ is $G/N$-coflasque, so by the  induction hypothesis on $\text{rank}(U)$ we get $K$ an $RG$-permutation module, so that $U$ is an $RG$-permutation module by Lemma \ref{permuExt=0}. 
    \end{enumerate}
        
         
    \end{proof}
    \section{Counterexamples}\label{s4}
    Indeed, the conditions of Theorem \ref{main} cannot be naively weakened in general. In  joint work with MacQuarrie,  we show that the condition $(U/U^N)_N$ be an $RG$-permutation block module cannot be removed in general \cite[Examples 1 and 2]{J.M}. Examples showing that the  conditions $U^N$ being $G/N$-coflasque and $U_N$ being an $RG$-permutation module cannot be removed in general are given in \cite[Examples 1 and 2]{John}. When $|N|=p$ and $p$ is prime in $R$, the condition that $U\res{N}$ be an $RN$-permutation module is equivalent to $U_N$ being an $R$-lattice \cite[Theorem A]{TORRECILLAS2013533}. But in general, conditions \ref{it2} and \ref{it3} do not ensure that $U\res{N}$ is an $RN$-permutation module as the following example shows. It is obtained  using a tool invented by M.C.R. Butler  in \cite{butler2}, and further developed to construct examples of this type  in \cite{J.M}.
     \begin{example}\label{exe}
      Let us consider $R=\mathbb{Z}_2$ and $G=\langle n\rangle\times \langle c\rangle\times \langle m\rangle$, the elementary abelian $2$-group of order $8$.
    Let $U$ be the $\Z_2G$-lattice defined by the following action of $G$: 
    $$n=\begin{psmallmatrix}
            1 & 0 & 0 & 0 & 0 & 0 & 0 & 0 & 0 \\
            0 & 1 & 0 & 0 & 0 & 0 & 0 & 0 & 0 \\
            0 & 0 & 1 & 0 & 0 & 0 & 0 & 0 & 0 \\
            0 & 0 & 0 & 1 & 0 & 0 & 0 & 0 & 0 \\
            0 & 0 & 0 & 0 & 1 & 0 & 0 & 0 & 0 \\
            -1 & -1 & -1 & 0 & 0 & -1 & 0 & 0 & 0 \\
           -1 & 0 & 0 & 0 & 0 & 0 & -1 & 0 & 0 \\
           0 & -1 & 0 & 0 & 0 & 0 & 0 & -1 & 0 \\
           0 & 1 & 0 & -1 & -1 & 0 & 0 & 0 & -1
\end{psmallmatrix}, \;c=\begin{psmallmatrix}
    1 & 0 & 0 & 0 & 0 & 0 & 0 & 0 & 0 \\
0 & 1 & 0 & 0 & 0 & 0 & 0 & 0 & 0 \\
0 & 0 & 1 & 0 & 0 & 0 & 0 & 0 & 0 \\
-1 & -1 & 0 & -1 & 0 & 0 & 0 & 0 & 0 \\
-1 & 0 & -1 & 0 & -1 & 0 & 0 & 0 & 0 \\
0 & 0 & 0 & 0 & 0 & 1 & 0 & 0 & 0 \\
0 & 0 & 0 & 0 & 0 & 0 & 1 & 0 & 0 \\
0 & -1 & 0 & 0 & 0 & 0 & 0 & -1 & 0 \\
0 & 1 & 0 & 0 & 0 & -1 & -1 & 0 & -1
\end{psmallmatrix},$$
$$m=\begin{psmallmatrix}
    1 & 0 & 0 & 0 & 0 & 0 & 0 & 0 & 0 \\
0 & 1 & 0 & 0 & 0 & 0 & 0 & 0 & 0 \\
0 & 0 & 1 & 0 & 0 & 0 & 0 & 0 & 0 \\
0 & 0 & 0 & 1 & 0 & 0 & 0 & 0 & 0 \\
-1 & 0 & -1 & 0 & -1 & 0 & 0 & 0 & 0 \\
0 & 0 & 0 & 0 & 0 & 1 & 0 & 0 & 0 \\
-1 & 0 & 0 & 0 & 0 & 0 & -1 & 0 & 0 \\
0 & 0 & 0 & 0 & 0 & 0 & 0 & 1 & 0 \\
0 & 0 & 0 & -1 & 0 & -1 & 0 & -1 & -1
\end{psmallmatrix}.$$
Let $N=\langle n,c\rangle$. We have  $U^N$ a $\Z_2G$-permutation module and $U_N$  a $\Z_2G$-permutation module. The reader may check this by confirming that $(U^N)_{G/N}$, $U_N$ and $(U_N)_{G/N}$ are free $\Z_2$-modules and use \cite[Lemma 8]{John}. The module $(U/U^N)_N$ has trivial action, as may be seen from the fact that $I_{\langle m\rangle}U\leqslant I_NU+U^N$. However, $(U^{\langle n\rangle})_{N/\langle n \rangle}$  is not a free $\Z_2$-module, so $U^{\langle n\rangle}$ is not a $\Z_2[N/\langle n\rangle]$-permutation module by \cite[Lemma 8]{John}. Therefore, $U\res{N}$ is not a $\Z_2N$-permutation module.

    \end{example} 
    \printbibliography
\end{document}